\documentclass{birkjour}
\usepackage{comment}
 \newtheorem{theorem}{Theorem}[section]
 
 \newtheorem{lemma}[theorem]{Lemma}
 \newtheorem{proposition}[theorem]{Proposition}
 \theoremstyle{definition}
 
 \theoremstyle{remark}
 
 \newtheorem*{example}{Example}
 \numberwithin{equation}{section}
 
 \newcommand{\R}{\mathbb{R}}

 \newcommand{\M}{\mathbf{M}}

\begin{document}

%
%
%
%
%
%
%
%
%

\title[Volume Distance to Hypersurfaces]
 {Volume Distance to Hypersurfaces: \newline Asymptotic Behavior of its Hessian }

\author[M.Craizer]{Marcos Craizer}

\address{%
Departamento de Matem\'{a}tica- PUC-Rio\br
Rio de Janeiro\br
Brazil}

\email{craizer@puc-rio.br}

\thanks{The authors want to thank CNPq for finnancial support during the preparation of this manuscript.}
\author[R.C.Teixeira]{Ralph C. Teixeira}
\address{Departamento de Matem\'{a}tica Aplicada- UFF \br
Rio de Janeiro\br
Brazil}
\email{ralph@mat.uff.br}


\date{July 13, 2010}

\begin{abstract}
The volume distance from a point  $p$ to a convex hypersurface $M\subset\R^{N+1}$ is defined as the minimum $(N+1)$-volume
of a region bounded by $M$ and a hyperplane $H$ through the point. This function is differentiable in a neighborhood of $M$ and if we restrict its
hessian to the minimizing hyperplane $H(p)$ we obtain, after normalization, a symmetric bi-linear form $Q$.

In this paper, we prove that $Q$ converges to the affine Blaschke metric when we approximate the hypersurface along a curve whose points are centroids of parallel sections. 
We also show that the rate of this convergence is given by a bilinear form associated with the shape operator of 
$M$. These convergence results provide a geometric interpretation of the Blaschke metric and the shape operator in terms of the volume distance.
\end{abstract}

\maketitle

\section{Introduction}

Consider a strictly convex hypersurface $M\subset\R^{N+1}$,  a point $p$ in the convex side of $M$ and $n\in S^N$. Denote by $U(n,p)$ the region bounded by $M$ 
and a hyperplane $H(n,p)$ orthogonal to $n$ through $p$, with $n$ pointing outwards the region,
and by $V(n,p)$ its volume.  The {\it volume distance} $v(p)$ of $p$ to $M$ is defined as the minimum of $V(n,p)$, $n\in S^N$.

The volume distance is an important object in computer vision which has been extensively studied in the planar case $n=1$ (\cite{Giblin04}) and was also 
considered in the case $n=2$ (\cite{Betelu01}). For $n=1$, the hessian of the volume distance was studied in (\cite{Craizer08},\cite{Silva09}), where it is shown that its determinant equals $-1$. 
This property is not extended to higher dimensions. Nevertheless, we prove in this paper some asymptotic properties of the hessian of the volume distance in arbitrary dimensions.

A pair $(n,p)$ is called {\it minimizing} when $n$ is the minimum of $V(n,p)$ with $p$ fixed. A minimizing pair necessarily satisfies
\begin{equation}\label{eq:critical}
\frac{\partial V}{\partial n}(n,p)=0.
\end{equation}
It is proved in \cite{Leichtweiss98} that if $(n,p)$ satisfies \eqref{eq:critical}, then $p$ is the centroid of $R(n,p)$.

In order to obtain $n=n(p)$ implicitly defined by \eqref{eq:critical}, the second derivative of $V$ with respect to $n$ must be non-degenerate.  A formula for this second derivative
can also be found in \cite{Leichtweiss98}. 
From this formula, one concludes that the second derivative is positive definite in a half-neighborhood of $M$, i.e., the part of a neighborhood of $M$ contained in its convex side. 
Based on this, we verify that there exists a half-neighborhood 
$D$ of $M$ such that, for any $p\in D$, there exists a unique $n(p)$ that minimizes the map $n\to V(n,p)$. 
Moreover, the map $p\to n(p)$ is smooth and consequently $v(p)=V(n(p),p)$ is also smooth.

For $p\in D$, let 
\begin{equation}\label{eq:DefineQ}
Q(p)=\frac{1}{b(p)} \frac{\partial^2V}{\partial n^2}(n(p),p),
\end{equation}
where $b(p)$ denotes the $N$-dimensional volume of the region $R(p)\subset H(p)$ bounded by $M$. 
By making some calculations, we show that, for  $p\in D$,  
\begin{equation}\label{eq:CalculaQ}
-\frac{1}{b(p)} \left. D^2v(p) \right|_{H(p)}= Q^{-1}(p)
\end{equation}
where $ \left. D^2v(p) \right|_{H(p)}$ means the restriction of $D^2v(p)$ to $H(p)$.

This paper is concerned with the asymptotic behavior of the quadratic form $Q$. In order to motivate a bit more this study, we remark that this quadratic form is an important tool in the study of 
 {\it floating bodies}. 
When $M$ is the boundary of a convex body $K$,  one can define its floating body $K_{\delta}$, for $\delta>0$, by the property that each support hyperplane of $K_{\delta}$ cuts $K$ in a region of volume $\delta$. 
For smooth strictly convex bodies and $\delta$ sufficiently small, the convex bodies exist and 
its boundary is a smooth surface (see \cite{Leichtweiss98}).  In \cite{Meyer91}, the quadratic form $Q$ was a key ingredient in proving that $K_{\delta}$ is well defined for every $0<\delta\leq \frac{1}{2}vol(K)$ if and only if $K$ is symmetric 
with respect to a point. 
Also in \cite{Schutt94}, $Q$ appears as a tool in proving that a convex body with a sequence of homothetic floating bodies must be an ellipsoid.

For $q\in M$, denote by $T_{q}M=H(n(q),q)$ the tangent plane to $M$ at $q$ and, for $t>0$, define $\gamma_q(t)$ as the centroid of the region $R(n(q),q+t\xi(q))$, where $\xi(q)$
is the affine normal to $M$ at $q$.  
We shall consider two symmetric bilinear forms defined on $T_{q}M$: the Blaschke metric $h$ which is positive definite and $h_S$ defined as $h_S(X,Y)=h(X,SY)$,
where $S$ is the shape operator. By identifying $H(\gamma_q(t))$ with $T_qM$, the normalized hessian $Q(\gamma_q(t))$ can also be seen as a symmetric bilinear form in $T_{q}M$. 
The  main result of the paper says that 
$$
Q(\gamma_q(t))=h(q)+th_S(q)+O(t^2),
$$
where $O(t^{k})$ indicates a quantity such that $\lim_{t\to 0}\frac{O(t^{k})}{t^{k-\epsilon}}=0$, for any $\epsilon>0$. This result can be regarded as a geometric interpretation of  the Blaschke metric and the shape operator in terms of the volume distance.

\medskip
\paragraph*{Acknowledgements.}
The authors want to thank Professor Peter J.Giblin for stimulating discussions during the preparation of this paper. 

\section{Hessian of the volume distance}

\subsection{Notation}

Consider a strictly convex hypersurface $M\subset\R^{N+1}$, possibly with a non-empty boundary $\partial M$. Denote by $H(n,p)\subset \R^{N+1}$ the hyperplane passing through $p\in\R^{N+1}$ with 
normal $n\in S^N$. 
For $p\in\R^{N+1}$, denote by $E(p)\subset S^N$ the set of unitary vectors $n$ whose corresponding hyperplane $H(n,p)$ intersects $M-\partial M$ transversally at a closed hypersurface $\Gamma(n,p)\subset H(n,p)$ 
bounding a region $R(n,p)\subset H(n,p)$ containing $p$ in its interior  
and such that the region $U(n,p)$ bounded by $R(n,p)$ and $M$, with $n$ pointing outwards, has finite volume $V(n,p)$ (see figure \ref{regionU}). 
Denote by $D_1\subset\R^{N+1}$ the set of $p\in\R^{N+1}$ such that $E(p)\neq\emptyset$  and the infimum 
$\inf\{V(n,p)|\ n\in E(p)\}$ is attained at $E(p)$. When $n\in E(p)$ attains this minimum, we call the pair $(n,p)$ minimizing and $v(p)=V(n,p)$ the volume distance to $M$. 
We remark that if $M$ is a closed hypersurface enclosing a convex region, then the domain $D_1$ of the volume distance is all the enclosed region. 

\begin{figure}[htb]
 \centering
 \includegraphics[width=0.40\linewidth]{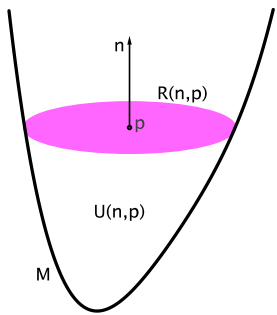}
 \caption{ The section $R(n,p)$ and the enclosed region $U(n,p)$.}
\label{regionU}
\end{figure}

For $q\in M$, denote by $\xi(q)$ the affine normal vector pointing to the convex side of $M$. Along this paper, we shall call a  half-neighborhood of $M$ any set of the form
$$
\{q+t\xi(q)|\ q\in M, 0\leq t<T(q)\},
$$
where $T(q)>0$ is some smooth function of $q$.

Close to a pair $(n_0,p_0)$,
consider cartesian coordinates $(x,z)\in \R^N\times I$, $I=(-\epsilon,\epsilon)$ such that $p_0=(0,0)$ and $n_0=(0,1)$.  To describe the hypersurface $M$ in a neighborhood of 
$H(n_0,p_0)$, 
consider cylindrical coordinates $(r,\eta,z)$, where $x=r\eta$, $\eta\in S^{N-1}$, $r>0$. Then
 $M$ is described by $r=r(\eta,z)$, for some smooth function $r$ (see figure \ref{sectionrz}). We write 
\begin{equation}\label{EqM}
r(\eta,z)=r(\eta,0)+r_z(\eta,0)z+O(z^2),
\end{equation}
for $z$ close to $0$. 

\begin{figure}[htb]
 \centering
 \includegraphics[width=0.50\linewidth]{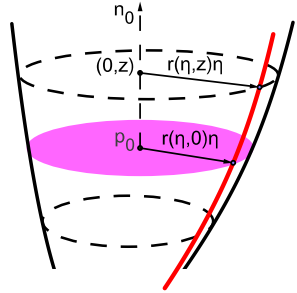}
 \caption{The curve $r=r(\eta,z)$ with fixed $\eta\in S^N$.}
\label{sectionrz}
\end{figure}

\subsection{Smoothness of the volume distance $v$ in a half-neighborhood of $M$ }

The derivative $\frac{\partial V}{\partial n}(n,p_0)$ can be regarded as a linear functional on $T_nS^{N}$, which can be identified with $H(n,p_0)$. 
The proof of next proposition can be found in \cite{Leichtweiss98}, p. 166. 

\begin{proposition}  
Denote by $\overline{p}(n,p)$ the center of gravity of $R(n,p)$ and by $b(n,p)$ the $N$-dimensional volume of the region $R(n,p)$. 
Then 
\begin{equation}\label{centroide}
\frac{\partial V}{\partial n}(n,p)=-b(n,p)\left(\overline{p}(n,p)-p\right).
\end{equation}
Thus, a pair $(n,p)$ is critical if and only if  $\overline{p}(n,p)=p$. 
\end{proposition}

The second derivative $\frac{\partial^2V}{\partial n^2}(n,p)$ can be seen as a linear operator of $T_nS^{N}$. Next proposition, whose proof can be found in \cite{Leichtweiss98}, p. 168, describe
this linear operator in the above defined cylindrical coordinates.

\begin{proposition}
Denote
$\M_N$ the symmetric positive definite $N\times N$ matrix $\eta\cdot\eta^t$, where $\eta$ is a column vector 
and $\eta^t$ its transpose. 
We have that 
\begin{equation}\label{d2vdn2}
\frac{\partial^2V}{\partial n^2}(n_0,p_0)=\int_{S^{N-1}}r^{N+1}(\eta,0)r_z(\eta,0)\M_N d\eta.
\end{equation}
\end{proposition}

If $r_z(\eta)>0$, for any $\eta\in S^{n-1}$, then formula \eqref{d2vdn2} implies $\frac{\partial^2 V}{\partial n^2}(n_0,p_0)$ is positive definite. Based on this, we can prove 
the following proposition:

\begin{proposition}\label{smoothD}
There exists a half-neigborhood $D\subset D_1$ of $M$ such that for any $p\in D$ there exists a smooth function $n=n(p)$ such that the pair $(n(p),p)$ is minimizing
and  $\frac{\partial^2 V}{\partial n^2}(n(p),p)$ is positive definite. 
\end{proposition}

\begin{proof} 
Given $q\in M$ consider a neighborhood $W$ of $q$ in $M$ with the following property: for any pair $(n,p)$ such that $\Gamma(n,p)\subset W$, $r_z(n,p)$ is strictly positive. 
For $p$ fixed, denote by $E_1(p)=\{n\in S^{N-1}|\ \Gamma(n,p)\subset W\}$. 

There is a half-neighborhood $U(q)$ of $q$  such that for any $p\in U(q)$, there exists a minimizing $n(p)\in E_1(p)$ and any minimizing pair
$n(p)$ must be in $E_1(p)$. Since $r_z(n,p)$ is strictly positive, the map $n\in E_1(p)\to V(n,p)$ is convex, so the minimizer $n(p)$ is unique. Considering $D=\cup_{q\in M-\partial M}U(q)$, 
we complete the proof of the proposition.
\end{proof}

\subsection{Derivatives of the volume distance}

Consider $D$ the half-neighborhood of $M$ given by proposition \ref{smoothD} and let $p\in D$. 
Recall that 
\begin{equation}\label{eq:Definev}
v(p)=V(n(p),p).
\end{equation}

\begin{lemma}
We have that 
\begin{equation}\label{eq:dVdp}
\frac{\partial V}{\partial p}(n,p)= b(n,p)n.
\end{equation}
As a consequence, 
\begin{equation}\label{eq:dvdp}
Dv(p)= b(n(p),p)n(p).
\end{equation}
\end{lemma}
\begin{proof}
Since $p\to V(n,p)$ is constant along the hyperplane $H(n,p)$, we conclude that $\frac{\partial V}{\partial p}(n,p)$ is parallel to
$n$. Also, for $t$ small, 
$$
V(n,p+tn)-V(n,p)=tb(n,p)+O(t^2),
$$
and thus the first formula is proved. Now differentiating \eqref{eq:Definev} we obtain \eqref{eq:dvdp}.
\end{proof}

\begin{proposition}
The normalized hessian of $v$ is exactly $Q^{-1}$, i.e., 
\begin{equation*}
-\frac{1}{b(p)} \left. D^2v(p) \right|_{H(p)}=Q^{-1}.
\end{equation*}
\end{proposition}

\begin{proof}
Differentiating \eqref{eq:dvdp} with respect to $p$ and using that $n$ is orthogonal to $H(p)$, we obtain
\begin{equation*}
 \left. D^2v(p) \right|_{H(p)}=b(p)\left. \frac{dn}{dp} \right|_{H(p)}.
\end{equation*} 
On the other hand, if we differentiate \eqref{eq:critical} with respect to $p$ we obtain 
\begin{equation*}
\frac{\partial^2 V}{\partial n^2}(n,p)\frac{dn}{dp}+\frac{\partial^2 V}{\partial n\partial p}=0.
\end{equation*}
Now, from \eqref{eq:dVdp}, 
\begin{equation*}
\frac{\partial^2 V}{\partial n\partial p}=b(p)I+\frac{\partial b}{\partial n}n.
\end{equation*}
We conclude that
\begin{equation*}
\left. \frac{dn}{dp} \right|_{H(p)}=-b(p)\left[\frac{\partial^2 V}{\partial n^2}(n,p)\right]^{-1},
\end{equation*}
thus proving the proposition. 
\end{proof}

\section{Convergence to the Blaschke metric}

For $q\in M$, consider the centroid $\gamma_q(t),\ t>0$ of the region $R(n(q),q+t\xi(q))$, where $n(q)$ is orthogonal to $T_qM$ and $\xi(q)$ is the affine normal vector at $q$.
Then $Q(\gamma_q(t))$ is a symmetric bilinear form defined in $H(\gamma_t(q))$, which can be identified with $T_{q}M$. The aim of this section is to prove the following theorem:

\begin{theorem} \label{Asymptotic1}
For $q\in M$, 
\begin{equation}
Q(\gamma_q(t))=h(q)+O(t),
\end{equation}
and so $Q(\gamma_q(t))$ is converging to $h(q)$ when $t$ goes to $0$. 
\end{theorem}

\medskip\medskip

By applying a suitable affine transformation, we may assume that \newline $q=(0,0)$, the tangent plane $T_qM$ is $z=0$ and the affine normal at $q$ is $(0,1)$. 
Then, close to $q$,  the surface $M$ is defined by an equation of the form
\begin{equation}\label{Normal}
z=\frac{r^2}{2}+O(r^3).
\end{equation}
where $O(r^k)$ may depend on $\eta$ but satisfies $\lim_{r\to 0}\frac{O(r^k)}{r^{k-\epsilon}}=0$, for any $\epsilon>0$. 
In this coordinates $h(q)=I$ and $\xi(q)=(0,1)$. Thus we can choose $t=z$ and write $\gamma_q(z)=(\overline{x}(z),z)$.

The following lemma is the main tool for proving theorem \ref{Asymptotic1}:

\begin{lemma}\label{AsymptoticQ1}
Define 
\begin{equation}\label{eq:DefineQ1}
Q_1(z)=  \frac{1}{b(z)}\int_{S^{N-1}}r^{N+1}(\eta,z)r_z(\eta,z)  \M_N(\eta) d\eta,
\end{equation}
where $b(z)$ denotes the $N$-volume of the section parallel to the hyperplane $z=0$ at height $z$. Then 
\begin{equation*}
Q_1(z)=I+O(z).
\end{equation*}
\end{lemma}

We now show how theorem \ref{Asymptotic1} follows from lemma \ref{AsymptoticQ1}. Since $\xi(q)$ is tangent to the centroid line (\cite{Simon93}, p.52), we have that $\overline{x}(z)=O(z^2)$. 
Now from equations \eqref{eq:DefineQ} and \eqref{d2vdn2} we conclude that $Q(\gamma_q(z))$ is $O(z^2)$-close to $Q_1(z)$.
Hence lemma \ref{AsymptoticQ1} implies that 
$$
Q(\gamma_q(z))=I+O(z),
$$
thus proving theorem  \ref{Asymptotic1}.

It remains then to prove lemma \ref{AsymptoticQ1}. 

\begin{proof} 
Since
$\lim_{r\to 0}\frac{r}{\sqrt{2}z^{1/2}}=1,$
we can write 
\begin{equation}\label{r(z)}
r(\eta,z)=\sqrt{2}z^{1/2}+O(z^{3/2}).
\end{equation}
Straightforward calculations from \eqref{r(z)} show that
$$
\frac{r^N}{N2^{N/2}}=\frac{1}{N}z^{N/2}+O(z^{N/2+1}).
$$
Differentiating $\frac{r^{N+2}}{N+2}$ with respect to $z$ leads to
\begin{equation*}
 \frac{r^{N+1}r_z}{2^{N/2}}=z^{N/2}+O(z^{N/2+1}).
\end{equation*}
The integral of $\eta_i\eta_j$ over $S^{N-1}$ is equal to
$\frac{\lambda}{N}\delta_{ij}$, where $\lambda=\lambda(N)$ is the Lebesgue measure of $S^{N-1}$ and $\delta_{ij}=1$, if $i=j$, and $0$, if $i\neq j$. Thus
the integral $L(i,j)$ of $r^{N+1}r_z\eta_i\eta_j$ satisfies
$$
\frac{L(i,j)}{2^{N/2}}= \frac{\lambda\delta_{ij}}{N}z^{N/2}+O(z^{N/2+1}).
$$
Also, calculating $b(z)$ as the integral of $r^N/N$ over $S^{N-1}$ we obtain 
$$
\frac{b(z)}{2^{N/2}}=\frac{\lambda}{N}z^{N/2}+O(z^{N/2+1}).
$$
Thus
$$
2^{N/2}b(z)^{-1}=\frac{N}{\lambda}z^{-N/2}+O(z^{-N/2+1}).
$$
and so $Q(z)(i,j)=b(z)^{-1}L(i,j)=\delta_{ij}+O(z)$. 
\end{proof}

\section{ Convergence to the shape operator }

Along this section, we shall use the notation of \cite{Nomizu94}: let $f:M\subset\R^N\to\R^{N+1}$ be the inclusion map and denote by  $\xi$ its normal vector field
pointing to the convex part of $M$. For $X,Y\in {\mathcal X}(U)$, we write
\begin{eqnarray*}
D_Xf_*(Y)&=&f_*(\nabla_XY)+h(X,Y)\xi\\
D_X\xi&=&-f_*(SX),
\end{eqnarray*}
where $\nabla$ denotes the Blaschke connection, $h$ is the positive definite Blaschke metric and $S$ is the shape operator. 
Denote by $\nu:M\to\R_{N+1}$ the corresponding co-normal immersion. 

Close to the hypersurface $M$, we write $p=\gamma_q(t)$, $q\in M$, $t\in[0,T)$, where $\gamma_q(t)$ is the centroid of the section through $q+t\xi(q)$ parallel to $T_qM$. 
Then $p$ is not necessarily on the normal line $q+t\xi(q)$, but we can write 
\begin{equation}\label{DefCent}
p=q+t\xi(q)+Z,
\end{equation}
for some $Z=Z(q,t)\in T_qM$,  with $Z=O(t^2)$ (see \cite{Simon93}, p.52). 
Differentiating  \eqref{DefCent} with respect to $t$ gives
\begin{equation}\label{Derivt}
\frac{\partial p}{\partial t}=\xi(q)+Z_t,
\end{equation}
for some $Z_t\in T_qM$, with $Z_t=O(t)$. 
We conclude that
\begin{equation*}
v_t(p)=Dv(p)\cdot\left(\xi(q)+Z_t\right)=Dv(p)\cdot\xi(q),
\end{equation*}
where for the last equality we have used the orthogonality of $Dv(p)$ and $H(p)$ (see equation \eqref{eq:dvdp}).
We have thus proved the following lemma:

\begin{lemma}
The derivative of $v$ is given by 
\begin{equation}\label{GradVCent}
Dv(p)=v_t(p)\ \nu(q),
\end{equation}
where $\nu(q)$ is the co-normal vector at $q\in M$ and $v_t(p)=\frac{d}{dt} v(\gamma_q(t))$.
\end{lemma}

\begin{lemma}\label{diagonal} For any $X\in T_qM$, 
\begin{equation*}
\lim_{t\to 0}\frac{1}{v_t}\cdot D^2v(X,\xi)=0.
\end{equation*}
\end{lemma}
\begin{proof}
Differentiate  equation \eqref{GradVCent} with respect to $t$ and use \eqref{Derivt} to obtain
\begin{equation*}
D^2v(\xi(q)+Z_t)=v_{tt}\nu(q).
\end{equation*}
Thus, for any $X\in T_qM$, 
$$
D^2v(\xi(q)+Z_t,X)= 0.
$$
So $D^2v(X,\xi)=-D^2v(X,Z_t)$ and hence
\begin{equation*}
\frac{1}{v_t}\cdot D^2v(X,\xi)=Q(\gamma_q(t))(X,Z_t).
\end{equation*}
By corollary \ref{Asymptotic1}, $Q(\gamma_q(t))$ is converging to $h$ and since $Z_t=O(t)$, we conclude that this last expression converges to $0$, thus proving the lemma.
\end{proof}

\begin{theorem}\label{Asymptotic2} The rate of convergence of the bi-linear form $Q(\gamma_q(t))$ to $h(q)$ is $h_S(q)$, i.e., 
$$
\lim_{t\to 0}\frac{Q(\gamma_q(t))(X,Y)-h(q)(X,Y)}{t}=h_S(q)(X,Y).
$$
for any $q\in M$, $X,Y\in T_qM$.
\end{theorem}

\begin{proof}
Observe first that if we differentiate \eqref{DefCent} in the direction $X\in T_qM$, we obtain
\begin{equation}\label{DerivX}
D_X(p)=(I-tS)X+\nabla_XZ+h(X,Z)\xi(q),
\end{equation}
with $\nabla_XZ=O(t^2)$ and $h(X,Z)=O(t^2)$.
Then differentiate  equation \eqref{GradVCent} in the direction of $X\in T_xM$ to obtain
\begin{equation*}
D^2v(D_X(p))=v_t\nu_X(q)+ X(v_t) \nu(q).
\end{equation*}
Thus, for $Y\in T_qM$, 
\begin{equation*}
D^2v(D_X(p),Y)=v_t\nu_X(q)(Y)=-v_th(X,Y)
\end{equation*}
(see \cite{Nomizu94}, p.57, for the last equality).
Expanding this equation using \eqref{DerivX} and dividing by $v_t$ we obtain
$$
Q(\gamma_q(t))(I-tSX,Y)-h(X,Y)=-Q(\gamma_q(t))(\nabla_XZ,Y)+h(X,Z)\frac{D^2v(\xi,Y)}{v_t}.
$$
Now, from lemma \ref{diagonal} and theorem \ref{Asymptotic1}, we conclude that
$$
\lim_{t\to 0}\frac{Q(\gamma_q(t))(X,Y)-h(X,Y)}{t}=h(SX,Y),
$$
thus proving the theorem. 
\end{proof}

\begin{example}
Consider the surface $M\subset\R^3$ described by the equation
$$
z=\frac{1}{2}\left(x^2+y^2\right)+\frac{c}{6}\left(x^3-3xy^2\right)+\frac{1}{24}(a_{40}x^4+4a_{31}x^3y+6a_{22}x^2y^2+4a_{13}xy^3+a_{04}y^4).
$$
For this surface $\xi(0,0)=(0,0,1)$ and we write
\begin{equation*}
z=\frac{r^2}{2}+\frac{r^3}{6}P_3(\theta)+\frac{r^4}{24}P_4(\theta),
\end{equation*}
where $\eta=(\cos(\theta),\sin(\theta)$,
\begin{equation*}
P_3(\theta)=c\left(\cos^3\theta-3\cos\theta\sin^2\theta\right)=c\cos(3\theta)
\end{equation*}
and
\begin{equation*}
P_4(\theta)=a_{40}\cos^4\theta+4a_{31}\cos^3\theta\sin\theta+6a_{22}\cos^2\theta\sin^2\theta+4a_{13}\cos\theta\sin^3\theta+a_{04}\sin^4\theta.
\end{equation*} 
It is not difficult to show that, in a neighborhood of $(0,0)$, the inverse function $r=r(z)$ satisfies 
\begin{equation*}
r(\theta,z)=\sqrt{2}z^{1/2} -\frac{P_3(\theta)}{3}z+\frac{5P_3^2(\theta)-3P_4(\theta)}{18\sqrt{2}} z^{3/2}+O(z^2).
\end{equation*}
From this equation, long but straightforward calculations show that $Q(z)=I+zA+O(z^2)$, where
\[
A=\left[
\begin{array}{ll}
\frac{c^2}{2}-\frac{1}{4}(a_{40}+a_{22}) & -\frac{1}{4}(a_{31}+a_{13}) \\
 -\frac{1}{4}(a_{31}+a_{13})  & \frac{c^2}{2}-\frac{1}{4}(a_{22}+a_{04})
\end{array}
\right].
\]
On the other hand, we can calculate the shape operator of $M$ at the origin following \cite{Nomizu94}, p.47. 
In this way we verify that $h_S=-A$, in accordance with theorem  \ref{Asymptotic2}. 
\end{example}

\end{document}